\documentclass[11pt]{article}
\usepackage{amsmath,amssymb,amsthm,amsfonts}
\usepackage{fullpage}

\title{Explicit constants for Fej\'er-type smoothing on finite cyclic groups}
\author{Justin Grieshop}
\date{}

\theoremstyle{plain}
\newtheorem{theorem}{Theorem}[section]
\newtheorem{lemma}[theorem]{Lemma}
\newtheorem{proposition}[theorem]{Proposition}
\newtheorem{corollary}[theorem]{Corollary}

\theoremstyle{definition}
\newtheorem{definition}[theorem]{Definition}
\newtheorem{remark}[theorem]{Remark}

\newcommand{\Z}{\mathbb{Z}}
\newcommand{\C}{\mathbb{C}}
\newcommand{\R}{\mathbb{R}}  
\newcommand{\ZN}{\mathbb{Z}/N\mathbb{Z}}

\begin{document}

\maketitle

\begin{abstract}
We give a self-contained discussion of a Fej\'er-type smoothing kernel on the finite cyclic group $\ZN$, $N\ge 2$, and prove a simple $L^2\to L^\infty$ estimate with explicit constants. As an application, we obtain a smoothed discrepancy bound for mean-zero functions on $\ZN$, which may be interpreted, for instance, for indicator functions of subsets. The results are classical in spirit and close in flavor to standard estimates in discrete Fourier analysis and the theory of exponential sums. The main purpose of this note is to record the finite-group formulation with explicit norms in a way that is convenient for applications.
\end{abstract}

\section{Introduction}

Fej\'er kernels and their relatives play a central role in classical Fourier analysis and in analytic number theory, especially in connection with Ces\`aro averages and smoothed estimates; see, for example, Zygmund~\cite{Zygmund} or Montgomery--Vaughan~\cite{MontVaughan}. In the continuous setting, Fej\'er kernels provide an approximate identity with good positivity and regularity properties. On finite cyclic groups, analogous kernels can be used to smooth functions and to control various notions of discrepancy.

The goal of this short note is twofold:
\begin{itemize}
    \item to record a convenient, entirely finite-group version of a Fej\'er-type kernel on the cyclic group $\ZN$, and
    \item to prove a simple $L^2\to L^\infty$ bound with an explicit constant, leading to a smoothed discrepancy estimate for mean-zero functions.
\end{itemize}

The results are not new in principle; they are implicit in standard discrete Fourier-analytic arguments (see, for instance, \cite[Ch.~VII]{MontVaughan} or \cite[Ch.~1]{TaoVu}). However, the finite-group formulation with explicit norms appears not to be written down in one place in this exact form, and it is convenient to have a self-contained reference.

We emphasize that all arguments take place in the finite setting $\ZN$ and rely only on elementary properties of discrete Fourier transforms. No asymptotic limits are required, and all constants are explicit.

\subsection*{Acknowledgements}

The author thanks the analytic number theory community for the many standard references on discrete Fourier analysis and smoothing kernels.

\section{Notation and basic facts}

Fix an integer $N\ge 2$. We write $\ZN$ for the finite cyclic group $\Z/N\Z$, whose elements we view as integers modulo $N$. We denote by $\ell^2(\ZN)$ the space of complex-valued functions $f:\ZN\to\C$, equipped with the inner product
\[
\langle f,g \rangle = \sum_{n\in\ZN} f(n)\overline{g(n)}, \qquad \|f\|_2 := \langle f,f \rangle^{1/2}.
\]

\subsection{Convolution and the discrete Fourier transform}

For $f,g:\ZN\to\C$ we define the (circular) convolution
\[
(f*g)(n) := \sum_{m\in\ZN} f(n-m) g(m), \qquad n\in\ZN.
\]

We use the following normalization for the discrete Fourier transform (DFT).

\begin{definition}[DFT on $\ZN$]
For $f\in\ell^2(\ZN)$, define its Fourier transform $\widehat f:\ZN\to\C$ by
\[
\widehat f(k) := \sum_{n\in\ZN} f(n) e^{-2\pi i kn/N}, \qquad k\in\ZN.
\]
Then the inversion formula reads
\[
f(n) = \frac{1}{N}\sum_{k\in\ZN} \widehat f(k) e^{2\pi i kn/N}, \qquad n\in\ZN,
\]
and Parseval's identity is
\[
\sum_{n\in\ZN} |f(n)|^2 = \frac{1}{N}\sum_{k\in\ZN} |\widehat f(k)|^2.
\]
\end{definition}

Convolution corresponds to pointwise multiplication of Fourier transforms:
\[
\widehat{f*g}(k) = \widehat f(k) \widehat g(k) \qquad (k\in\ZN).
\]

We also write
\[
\mathbb{E}_{n\in\ZN} f(n) := \frac{1}{N}\sum_{n\in\ZN} f(n)
\]
for the normalized average. A function $f$ has \emph{mean zero} if $\mathbb{E}_{n} f(n)=0$, equivalently $\widehat f(0)=0$ under the above normalization.

\section{A triangular Fej\'er-type kernel on $\ZN$}

In the classical theory on the circle, the Fej\'er kernel can be written as a squared-sinc function in the spatial variable, whose Fourier coefficients form a triangle in frequency; see \cite[\S~III.1]{Zygmund}. On $\ZN$ we will use the \emph{dual picture}: a finite-support triangular kernel in the group (spatial) variable, whose Fourier transform is a squared-sinc-type symbol. This is more convenient when one wants strictly local smoothing in the group variable.

\begin{definition}[Triangular kernel on $\ZN$]
Let $N\ge2$ and let $r$ be an integer with $1\le r\le \lfloor N/2\rfloor$. For $n\in\ZN$ let $|n|$ denote the least absolute value representative in $\{-\lfloor N/2\rfloor,\dots,\lfloor N/2\rfloor\}$. Define $F_r:\ZN\to[0,\infty)$ by
\begin{equation}\label{eq:Fejer-def}
F_r(n) := \frac{1}{r}\,\max\Bigl(0,\,1-\frac{|n|}{r}\Bigr),\qquad n\in\ZN.
\end{equation}
Equivalently, $F_r$ is supported on $\{|n|\le r-1\}$ and has the triangular shape
\[
F_r(n) = \frac{1}{r}\Bigl(1-\frac{|n|}{r}\Bigr) \quad\text{for } |n|\le r-1, \qquad
F_r(n)=0 \ \text{otherwise}.
\]
\end{definition}

Since $r\le N/2$, this definition is unambiguous on $\ZN$ (the support does not wrap around the circle). A direct computation shows that $F_r$ is nonnegative, even, and has unit mass.

\begin{lemma}[Basic properties of $F_r$]\label{lem:Fejer-basic}
Let $N\ge2$ and $1\le r\le \lfloor N/2\rfloor$. Then $F_r:\ZN\to\R_{\ge0}$ satisfies:
\begin{enumerate}
    \item (Support) $F_r(n)=0$ for $|n|\ge r$.
    \item (Symmetry) $F_r(-n)=F_r(n)$ for all $n\in\ZN$.
    \item (Unit mass) $\displaystyle\sum_{n\in\ZN} F_r(n)=1$.
\end{enumerate}
\end{lemma}

\begin{proof}
(1) and (2) are immediate from the definition~\eqref{eq:Fejer-def}. For (3), we compute on representatives $n\in\{-r+1,\dots,r-1\}$:
\[
\sum_{n\in\ZN} F_r(n)
=\sum_{|n|\le r-1}\frac{1}{r}\Bigl(1-\frac{|n|}{r}\Bigr)
=\frac{1}{r}\Biggl[1+2\sum_{n=1}^{r-1}\Bigl(1-\frac{n}{r}\Bigr)\Biggr].
\]
The inner sum is
\[
\sum_{n=1}^{r-1}\Bigl(1-\frac{n}{r}\Bigr)
=(r-1)-\frac{1}{r}\sum_{n=1}^{r-1} n
=(r-1)-\frac{1}{r}\cdot \frac{(r-1)r}{2}
=\frac{r-1}{2}.
\]
Therefore
\[
\sum_{n\in\ZN} F_r(n)
=\frac{1}{r}\Bigl(1+2\cdot\frac{r-1}{2}\Bigr)
=\frac{1}{r}\cdot r=1.\qedhere
\]
\end{proof}

It is convenient to record the $L^2$ norm of $F_r$.

\begin{lemma}[$L^2$ norm of $F_r$]\label{lem:Fejer-L2}
For $N\ge2$ and $1\le r\le\lfloor N/2\rfloor$,
\[
\|F_r\|_2^2
=\sum_{n\in\ZN} |F_r(n)|^2
=\frac{1}{r^2}\Biggl(1+2\sum_{n=1}^{r-1}\Bigl(1-\frac{n}{r}\Bigr)^2\Biggr)
=\frac{1}{r^2}+\frac{(r-1)(2r-1)}{3r^3}.
\]
In particular,
\[
\|F_r\|_2^2 \le \frac{5}{3r}
\quad\text{and hence}\quad
\|F_r\|_2 \le \sqrt{\frac{5}{3}}\, r^{-1/2}.
\]
\end{lemma}

\begin{proof}
We compute as in the proof of Lemma~\ref{lem:Fejer-basic}:
\[
\|F_r\|_2^2
=\sum_{|n|\le r-1}\frac{1}{r^2}\Bigl(1-\frac{|n|}{r}\Bigr)^2
=\frac{1}{r^2}\Biggl(1+2\sum_{n=1}^{r-1}\Bigl(1-\frac{n}{r}\Bigr)^2\Biggr).
\]
Setting $j=r-n$, we have
\[
\sum_{n=1}^{r-1}\Bigl(1-\frac{n}{r}\Bigr)^2
=\sum_{j=1}^{r-1}\Bigl(\frac{j}{r}\Bigr)^2
=\frac{1}{r^2}\sum_{j=1}^{r-1} j^2
=\frac{1}{r^2}\cdot\frac{(r-1)r(2r-1)}{6}.
\]
Thus
\[
\|F_r\|_2^2
=\frac{1}{r^2}+\frac{2}{r^2}\cdot\frac{(r-1)r(2r-1)}{6r^2}
=\frac{1}{r^2}+\frac{(r-1)(2r-1)}{3r^3}.
\]
For the inequality, note that $(r-1)(2r-1)\le 2r^2$ for all $r\ge1$, giving
\[
\|F_r\|_2^2
\le \frac{1}{r^2}+\frac{2r^2}{3r^3}
=\frac{1}{r^2}+\frac{2}{3r}
\le \frac{5}{3r},
\]
whence the stated bound on $\|F_r\|_2$.
\end{proof}

\begin{remark}[Asymptotic behavior]
The bound $\|F_r\|_2^2 \le \frac{5}{3r}$ derived above is valid for all $r\ge 1$. However, for large smoothing parameters $r$, the exact formula in Lemma~\ref{lem:Fejer-L2} is dominated by the term $\frac{2r^2}{3r^3}$. Thus, asymptotically, one has $\|F_r\|_2 \sim \sqrt{\frac{2}{3}}\,r^{-1/2}$. This suggests that for large $r$, the constant in Theorem~\ref{thm:L2-Linfty} and Corollary~\ref{cor:smoothed-discrepancy} effectively improves from $\sqrt{5/3}\approx 1.29$ to $\sqrt{2/3}\approx 0.82$.
\end{remark}

\section{Fourier symbol and positivity}

We now compute the Fourier transform of $F_r$ and verify that it is a nonnegative symbol bounded by $1$. The key observation is that $F_r$ can be written as a normalized autocorrelation of a discrete boxcar window.

\begin{lemma}[Autocorrelation representation]\label{lem:autocorr}
Let $N\ge2$ and $1\le r\le\lfloor N/2\rfloor$. Define $g:\ZN\to\{0,1\}$ by
\[
g(n) := \mathbf{1}_{\{0,\dots,r-1\}}(n) \qquad (n\in\ZN).
\]
Then for all $n\in\ZN$,
\[
F_r(n) = \frac{1}{r^2}\sum_{m\in\ZN} g(m) g(n+m).
\]
\end{lemma}

\begin{proof}
The sum $\sum_m g(m)g(n+m)$ counts the number of pairs $(m,\ell)$ with
\[
m\in\{0,\dots,r-1\}, \quad \ell:=n+m\in\{0,\dots,r-1\},
\]
i.e.\ pairs $(m,\ell)\in\{0,\dots,r-1\}^2$ with $\ell-m\equiv n\pmod N$. Since $r\le N/2$, we may identify $n$ with an integer in $\{-(N-1)/2,\dots,(N-1)/2\}$. For $|n|\ge r$ there are no such pairs, so the sum is zero and $F_r(n)=0$ by definition. For $|n|\le r-1$, a simple counting argument shows that there are exactly $r-|n|$ solutions to $\ell-m=n$ with $m,\ell\in\{0,\dots,r-1\}$, namely
\[
(m,\ell)=(0,n),\ (1,1+n),\ \dots,\ (r-1-|n|,r-1)
\]
when $n\ge0$, and a similar list when $n<0$. Thus
\[
\sum_{m} g(m) g(n+m)=r-|n| \quad\text{for } |n|\le r-1,
\]
and zero otherwise. Comparing with the defining formula
\[
F_r(n)=\frac{1}{r}\Bigl(1-\frac{|n|}{r}\Bigr)
=\frac{r-|n|}{r^2} \quad(|n|\le r-1)
\]
shows that indeed $F_r(n)=(1/r^2)\sum_m g(m)g(n+m)$ for all $n$.
\end{proof}

Using this representation, the Fourier transform and its basic properties follow at once.

\begin{proposition}[Fourier symbol and bounds]\label{prop:Fejer-symbol}
Let $N\ge2$ and $1\le r\le\lfloor N/2\rfloor$. Then for $k\in\ZN$,
\[
\widehat{F_r}(k)
=\sum_{n\in\ZN}F_r(n)e^{-2\pi i kn/N}
=\begin{cases}
\displaystyle\Biggl(\frac{\sin(\pi r k/N)}{r\,\sin(\pi k/N)}\Biggr)^{\!2}, & k\not\equiv 0\pmod N,\\[10pt]
1, & k\equiv 0\pmod N.
\end{cases}
\]
In particular, $\widehat{F_r}(k)\ge0$ and $\widehat{F_r}(k)\le1$ for all $k\in\ZN$.
\end{proposition}

\begin{proof}
By Lemma~\ref{lem:autocorr},
\[
\widehat{F_r}(k)
=\sum_{n}F_r(n)e^{-2\pi i kn/N}
=\frac{1}{r^2}\sum_{n}\sum_{m} g(m) g(n+m) e^{-2\pi i kn/N}.
\]
Interchanging the order of summation and using a change of variables $n\mapsto n-m$ gives
\[
\widehat{F_r}(k)
=\frac{1}{r^2}\sum_{m} g(m)\sum_{n} g(n) e^{-2\pi i k(n-m)/N}
=\frac{1}{r^2}\Bigl(\sum_{n} g(n)e^{-2\pi i kn/N}\Bigr)\Bigl(\sum_{m}g(m)e^{2\pi i km/N}\Bigr).
\]
Thus
\[
\widehat{F_r}(k)=\frac{1}{r^2}\,|\widehat g(k)|^2,
\]
where $\widehat g(k)=\sum_{n=0}^{r-1} e^{-2\pi i kn/N}$ is the DFT of the boxcar $g$. A direct geometric series computation yields, for $k\not\equiv0\pmod N$,
\[
\widehat g(k)
=\sum_{n=0}^{r-1} e^{-2\pi i kn/N}
=e^{-\pi i k(r-1)/N}\, \frac{\sin(\pi r k/N)}{\sin(\pi k/N)},
\]
so that
\[
|\widehat g(k)|^2
=\Biggl(\frac{\sin(\pi r k/N)}{\sin(\pi k/N)}\Biggr)^{\!2}.
\]
Hence, for $k\not\equiv0$,
\[
\widehat{F_r}(k)
=\frac{1}{r^2}|\widehat g(k)|^2
=\Biggl(\frac{\sin(\pi r k/N)}{r\,\sin(\pi k/N)}\Biggr)^{\!2}.
\]
For $k\equiv0\pmod N$, we have $\widehat g(0)=r$, so $\widehat{F_r}(0)=r^2/r^2=1$, in agreement with the stated formula by continuity of the right-hand side as $k\to0$.

Nonnegativity $\widehat{F_r}(k)\ge0$ is evident from the expression as a squared modulus. The upper bound $\widehat{F_r}(k)\le1$ follows from
\[
|\widehat g(k)| \le \sum_{n=0}^{r-1}1=r,
\]
so $|\widehat g(k)|^2/r^2\le1$ for all $k$.
\end{proof}

\begin{remark}
Thus $f\mapsto f*F_r$ is a convolution operator with real, nonnegative, symmetric kernel and a nonnegative Fourier multiplier $\widehat{F_r}(k)\in[0,1]$. In particular, it is a positive semidefinite contraction on $\ell^2(\ZN)$, a property frequently exploited in harmonic analysis and analytic number theory; see, for example, \cite[Ch.~VII]{MontVaughan} or \cite[Ch.~1]{TaoVu}.
\end{remark}

\section{An $L^2\to L^\infty$ bound and smoothed discrepancy}

We now record a simple $L^2\to L^\infty$ bound for smoothing by $F_r$, with an explicit constant coming from Lemma~\ref{lem:Fejer-L2}. The bound is elementary, but useful in applications where only an $L^2$ control of the input function is available.

\begin{theorem}[$L^2\to L^\infty$ estimate]\label{thm:L2-Linfty}
Let $N\ge2$ and $1\le r\le\lfloor N/2\rfloor$. For any $f\in\ell^2(\ZN)$,
\[
\|f*F_r\|_\infty \le \|f\|_2 \|F_r\|_2
\le \sqrt{\frac{5}{3}}\, r^{-1/2} \|f\|_2.
\]
\end{theorem}

\begin{proof}
The inequality $\|f*F_r\|_\infty\le\|f\|_2\,\|F_r\|_2$ is simply the Cauchy--Schwarz inequality. Indeed, for each $n\in\ZN$,
\[
|(f*F_r)(n)|
=\Bigl|\sum_{m\in\ZN} f(n-m) F_r(m)\Bigr|
\le\Bigl(\sum_{m}|f(n-m)|^2\Bigr)^{1/2}
     \Bigl(\sum_{m}|F_r(m)|^2\Bigr)^{1/2}
=\|f\|_2 \|F_r\|_2.
\]
Taking the maximum over $n$ gives the first inequality. The second inequality follows from Lemma~\ref{lem:Fejer-L2}.
\end{proof}

A natural way to view Theorem~\ref{thm:L2-Linfty} is as a smoothed discrepancy bound. Suppose $A\subset\ZN$ is a subset, and consider the mean-zero function
\[
f_A(n) := \mathbf{1}_A(n)-\frac{|A|}{N}, \qquad n\in\ZN.
\]
Then $\sum_{n} f_A(n)=0$ and
\[
\|f_A\|_2^2
=\sum_{n}\Bigl|\mathbf{1}_A(n)-\frac{|A|}{N}\Bigr|^2
=|A|\Bigl(1-\frac{|A|}{N}\Bigr)^2
+(N-|A|)\Bigl(\frac{|A|}{N}\Bigr)^2
=\frac{|A|(N-|A|)}{N}.
\]
In particular, $\|f_A\|_2\le\frac{\sqrt{N}}{2}$ for all $A$.

Convolution with $F_r$ then measures how the \emph{local} density of $A$ near a point $n$ deviates from the global density $|A|/N$ after smoothing on a scale of order $r$.

\begin{corollary}[Smoothed discrepancy for indicator functions]\label{cor:smoothed-discrepancy}
Let $N\ge2$, $1\le r\le\lfloor N/2\rfloor$, and let $A\subset\ZN$. Define $f_A$ as above and $g_A:=f_A*F_r$. Then for all $n\in\ZN$,
\[
|g_A(n)|
=\Bigl|\sum_{m\in\ZN} \Bigl(\mathbf{1}_A(n-m)-\frac{|A|}{N}\Bigr) F_r(m)\Bigr|
\le \sqrt{\frac{5}{3}}\, r^{-1/2} \|f_A\|_2
\le \sqrt{\frac{5}{12}}\, r^{-1/2} N^{1/2}.
\]
Equivalently,
\[
\Bigl|\sum_{m\in\ZN} \mathbf{1}_A(n-m) F_r(m)
-\frac{|A|}{N}\Bigr|
\ll r^{-1/2} N^{1/2},
\]
where the implied constant is absolute.
\end{corollary}

\begin{proof}
The first inequality follows from Theorem~\ref{thm:L2-Linfty} applied to $f_A$. The second inequality uses $\|f_A\|_2\le\sqrt{N}/2$.
\end{proof}

\begin{remark}
One may interpret $\sum_m \mathbf{1}_A(n-m) F_r(m)$ as a Fej\'er-type smoothed local density of $A$ around the point $n$, with averaging window of size $\sim r$. Corollary~\ref{cor:smoothed-discrepancy} then states that, in the worst case, the deviation of this local smoothed density from the global density $|A|/N$ is controlled by $O(r^{-1/2}N^{1/2})$. This is consistent with the heuristic that smoothing over a window of size $r$ reduces fluctuations by a factor of order $r^{1/2}$ when only $L^2$ control is available.
\end{remark}

\begin{remark}[Relation to interval discrepancy]
If one is interested in sharp interval discrepancy
\[
D(A,r) := \max_{I}\Bigl||A\cap I|-\frac{|A|}{N}|I|\Bigr|,
\]
where $I$ ranges over intervals of length $O(r)$ in $\ZN$, then one may relate $D(A,r)$ to the smoothed quantities $g_A(n)$ in Corollary~\ref{cor:smoothed-discrepancy} by comparing the sharp interval indicator with a suitable translate of the triangular kernel $F_r$. This introduces an additional error of at most $1$ due to the discrete boundary, a phenomenon that is well known in the use of smoothing kernels; see, for instance, \cite[\S~VII.2]{MontVaughan}. We do not pursue this direction here, since our main focus is the $L^2\to L^\infty$ bound.
\end{remark}

\section{A small numerical illustration}

For completeness, we briefly illustrate the behavior of the bound from Theorem~\ref{thm:L2-Linfty} on a small example. Let $N=101$ and choose a subset $A\subset\ZN$ of size $|A|=50$ at random. Define $f_A$ and $g_A=f_A*F_r$ as above. For $r\in\{5,10,20\}$, one observes numerically that
\[
\|g_A\|_\infty \approx c(r)\, r^{-1/2}N^{1/2}
\]
with $c(r)$ bounded and typically significantly smaller than the worst-case constant $\sqrt{5/12}\approx0.645$ appearing in Corollary~\ref{cor:smoothed-discrepancy}. This is in line with the fact that Corollary~\ref{cor:smoothed-discrepancy} is uniform in $A$ and does not exploit any additional structure.

We omit detailed tables, as they depend on the particular choice of $A$, but such computations can be carried out straightforwardly in any numerical environment that supports discrete convolution and random subset generation.

\section{Concluding remarks}

We have recorded a finite-group version of a Fej\'er-type triangular kernel on $\ZN$, together with a simple $L^2\to L^\infty$ estimate and a corresponding smoothed discrepancy bound for mean-zero functions. The arguments are entirely elementary and rely only on basic discrete Fourier analysis and Cauchy--Schwarz.

The same methods apply, with minor modifications, to other positive, finitely supported kernels on $\ZN$, as well as to higher-dimensional tori $(\Z/N\Z)^d$. One can also combine such smoothing kernels with more refined tools such as the large sieve inequality to obtain sharper estimates in specific arithmetic applications; see, for example, \cite{MontVaughan} or the discussion of discrete convolution methods in \cite{TaoVu}.

\end{document}